\documentclass[11pt,a4paper]{amsart}
\usepackage{amsfonts}
\usepackage{epsfig}
\usepackage{graphicx}
\usepackage{amsmath}
\usepackage{amssymb}

\newcommand{\Rr}{\mathbb{R}}

\DeclareMathAlphabet{\mathpzc}{OT1}{pzc}{m}{it}

\newcounter{main}

\newtheorem{theorem}{Theorem}[section]
\newtheorem{proposition}[theorem]{Proposition}

\newtheorem{corollary}[theorem]{Corollary}

\newtheorem{maintheorem}{Theorem}

\newcommand{\blanksquare}{\,\,\,$\sqcup\!\!\!\!\sqcap$}

\newcounter{example}
{{\stepcounter{example}}{\flushleft {\bf Example \arabic{example}:}}}%
{\par}

\title[$C^1$-stably shadowable conservative diffeomorphisms are Anosov]{$C^1$-stably shadowable conservative\\ diffeomorphisms  are Anosov}

\author[M. Bessa]{M\'{a}rio Bessa}

\address{Universidade da Beira Interior, Rua Marqu\^es d'\'Avila e Bolama,
  6201-001 Covilh\~a
Portugal.}
\email{bessa@fc.up.pt}

\begin{document}

\begin{abstract}
In this short note we prove that if a symplectomorphism $f$ is $C^1$-stably sha\-dow\-a\-ble, then $f$ is Anosov.  The same result is obtained for volume-preserving diffeomorphisms.
\end{abstract}

\maketitle

\section{Introduction: basic definitions and statement of the results}

\begin{subsection}{The symplectomorphisms framework}

Denote by $M$ a $2d$-dimensional manifold with Riemaniann structure and endowed with a closed and nondegenerate 2-form $\omega$ called \emph{symplectic form}. Let $\mu$ stands for the volume measure associated to the volume form wedging $\omega$ $d$-times, i.e., $\omega^d=\omega\land\dots\land\omega$. The Riemannian structure induces a norm $\|\cdot\|$ on the tangent bundle $T_x M$. Denote the Riemannian distance by $d(\cdot,\cdot)$. We will use the canonical norm of a bounded linear map $A$ given by $\|A\|=\sup_{\|v\|=1}\|A\cdot v\|$. By the theorem of Darboux (see e.g.~\cite[Theorem 1.18]{MZ}) there exists an atlas $\{\varphi_j\colon U_j\to\Rr^{2d}\}$, where $U_j$ is an open subset of $M$, satisfying $\varphi_j^*\omega_0=\omega$ with $\omega_0=\sum_{i=1}^d dy_i\land dy_{d+i}$. A diffeomorphism $f\colon(M,\omega)\to(M,\omega)$ is called a \emph{symplectomorphism} if it leaves invariant the symplectic structure, say  $f^*\omega=\omega$. Observe that, since $f^*\omega^d=\omega^d$, a symplectomorphism $f\colon M\to M$ preserves the volume measure $\mu$.
Symplectomorphisms arise naturally in the rational mechanics formalism as the first return Poincar\'e maps of hamiltonian flows. For this reason, it has long been one of the most interesting research fields in mathematical physics. We suggest the reference \cite{MZ} for more  details on general hamiltonian and symplectic theories. Let $(\text{Symp}_{\omega}^1(M), C^1)$ denote the set of all symplectomorphisms of class $C^1$  defined on $M$, endowed with the usual $C^1$ Whitney topology.

\end{subsection}

\begin{subsection}{The shadowing property}

The concept of shadowing in dynamical systems is inspired by the numerical computational idea of estimating differences between exact and approximate solutions along orbits and to understand the influence of the errors that we commit and allow on each iterate.  We may ask if it is possible to obtain shadowing of approximate trajectories in a given dynamical system by exact ones. We refer Pilyugin's book ~\cite{P} for a completed description on the subject.

There are, of course, considerable limitations to the amount of information we can extract from a given specific system that exhibits the shadowing property, since a $C^1$-close system may be absent of that property. For this reason it is of great utility and natural to consider that a selected model can be slightly perturbed in order to obtain the same property - the stably shadowable dynamical systems. 

For $\delta>0$ and $a,b\in [-\infty,+\infty]$ such that $a<b$, the sequence of points $\{x_i\}_{i=a}^b$ in $M$ is called a \emph{$\delta$-pseudo orbit} for $f$ if $d(f(x_i),x_{i+1})<\delta$ for all $a\leq i \leq b-1$.

Let $\Lambda\subseteq M$ be a closed and $f$-invariant set. The symplectomorphism $f\colon M\rightarrow M$ is said to have \emph{the shadowing property} in $\Lambda$ if for all $\epsilon>0$, there exists $\delta>0$, such that for any $\delta$-pseudo orbit $(x_n)_{n\in\mathbb{Z}}$ in $\Lambda$, there is a point $x\in M$ which $\epsilon$-shadows $(x_n)_{n\in\mathbb{Z}}$, i.e. $d(f^i(x),x_i)<\epsilon$.

Let  $f \in \text{Symp}^{1}_\omega(M)$ we  say that $f$  is $C^1$-\emph{stably (or robustly) shadowable} if there exists a neighborhood $\mathcal{V}$ of $f$ in $\text{Symp}^{1}_\omega(M)$ such that any $g\in\mathcal{V}$ has the shadowing property. 

We point out that $f$ has the shadowing property if and only if $f^n$ has the shadowing property for every $n\in\mathbb{Z}$.

\end{subsection}

\begin{subsection}{Hyperbolicity and statement of the results}

We say that any element $f$ in the set $\text{Symp}_{\omega}^1(M)$ is \emph{Anosov} if, there exists $\lambda\in(0,1)$ such that the tangent vector bundle over $M$ splits into two $Df$-invariant subbundles $TM=E^u\oplus E^s$, with $\|Df^n|_{E^s}\|\leq \lambda^n$ and $\|Df^{-n}|_{E^u}\|\leq \lambda^n$. We observe that there are plenty Anosov diffeomorphisms which are not symplectic.

Let us recall that a periodic point $p$ of period $\pi$ is said to be \emph{hyperbolic} if the tangent map $Df^\pi(p)$ has no norm one eigenvalues.

Let  $f \in \text{Symp}^{1}_\omega(M)$ we  say that $f$ is in $\mathcal{F}^1_{\omega}(M)$ if there exists a neighborhood $\mathcal{V}$ of $f$ in $\text{Symp}^{1}_\omega(M)$ such that any $g\in\mathcal{V}$, has all the periodic orbits of hyperbolic type. We define analogously the set $\mathcal{F}^1_\mu(M)$ in the broader set of volume-preserving diffeomorphisms $\text{Symp}^{1}_\mu(M)$.

The results in this note can be seen as a generalization of the result in \cite{S} for symplectomorphisms and volume-preserving diffeomorphisms. Let us state our first result.

\begin{maintheorem}\label{teo1}
If a symplectomorphism $f$ defined in symplectic manifold $M$ is $C^1$-stably shadowable, then $f$ is Anosov.  
\end{maintheorem}

It is well known that Anosov diffeomorphisms impose severe topological restrictions to the manifold where they are supported. Thus, we present a simple but startling consequence of Theorem~\ref{teo1} that shows how topological conditions on the phase space imposes numerical restrictions to a given dynamical system.

\begin{corollary}
If $M$ does not support an Anosov diffeomorphisms, then there are no $C^1$-stably shadowable symplectomorphisms.
\end{corollary}

We end this introduction by recalling a result in the vein of ours proved in ~\cite{BR} - \emph{$C^1$-robust topologically stable symplectomorphisms are  Anosov}. 
We point out that, in the broader setting of dissipative diffeomorphisms,  it was proved, in \cite[Theorem 4]{W}, that \emph{expansiveness} and the (robust) shadowing property implies (robust) topologic stability. In  \cite[Theorem 11]{W} it is also proved that topological stability, imply the shadowing property. Another result which relates $C^1$-robust properties with hyperbolicity is the Horita and Tahzibi theorem (see~\cite{HT}) which states that $C^1$-robustly transitive symplectomorphisms are partially hyperbolic.

\bigskip

\end{subsection}

\section{Proof of Theorem~\ref{teo1}}\label{diff}

Theorem\ref{teo1} is a direct consequence of the following two propositions. The following result can be found in ~\cite{Ne}.

\begin{proposition}\label{Newhouse} (Newhouse)
If $f\in\mathcal{F}^1_\omega$, then $f$ is Anosov.
\end{proposition}

The following result is a symplectic version of  ~\cite[Proposition 1]{M}. Actually, Moriyasu, while working in the dissipative context, con\-si\-de\-red the shadowing property in the non-wandering set, which, in the symplectic setting, and due to Poincar\'e recurrence, is the whole manifold $M$.

\begin{proposition}\label{Moriyasu} 
If $f$ is a $C^1$-stably shadowable symplectomorphism, then $f\in\mathcal{F}^1_\omega$.
\end{proposition}

\begin{proof}

The proof is by contradiction; let us assume that there exists a $C^1$-stably shadowable symplectomorphism $f$ having a  non-hyperbolic closed orbit $p$ of period $\pi$.

In order to go on with the argument we need to $C^1$-approximate the symplectomorphism $f$ by a new one, $f_1$, which, in the local coordinates given by Darboux's theorem, is \emph{linear} in a neighborhood of the periodic orbit $p$. To perform this task, in the sympletic setting, and taking into account ~\cite[Lemma 3.9]{AM}, it is required higher smoothness of the symplectomorphism. 

Thus, if $f$ is of class $C^\infty$, take $g=f$, otherwise we use \cite{Z} in order to obtain a $C^\infty$ $C^1$-stably shadowable symplectomorphism $h$, arbitrarily $C^1$-close to $f$, and such that $h$ has a periodic orbit $q$, close to $p$, with period $\pi$. We observe that $q$ may not be the analytic continuation of $p$ and this is precisely the case when $1$ is an eigenvalue of the tangent map $Df^{\pi}(p)$. 

If $q$ is not hyperbolic take $g=h$. If $q$ is hyperbolic for $Dh^{\pi}(q)$, then, since $h$ is $C^1$-arbitrarily close to $f$, the distance between the spectrum of $Dh^{\pi}(q)$ and the unitary circle can be taken arbitrarily close to zero. This means that we are in the presence of ``poor" hyperbolicity, thus in a position to apply ~\cite[Lemma 5.1]{HT} to obtain a new $C^1$-stably shadowable symplectomorphism $g\in\text{Symp}_{\omega}^\infty(M)$, $C^1$-close to $h$ and such that $q$ is  a non-hyperbolic periodic orbit.

Now, we use the \emph{weak pasting lemma} (\cite[Lemma 3.9]{AM})  in order to obtain a $C^1$-stably shadowable symplectomorphism $f_1$  such that, in local canonical coordinates, $f_1$ is linear and equal to $Dg$ in a neighborhood of the periodic non-hyperbolic orbit, $q$. Moreover, the existence of an eigenvalue, $\lambda$, with modulus equal to one is associated to a symplectic invariant two-dimensional subspace contained in the subspace $E^c_q\subseteq T_q M$ associated to norm-one eigenvalues. Furthermore, up to a perturbation using again ~\cite[Lemma 5.1]{HT}, $\lambda$ can be taken rational. This fact assures the existence of periodic orbits arbitrarily close to the $f_1$-orbit of $q$. Thus, there exists $m\in\mathbb{N}$ such that $f_1^{m\pi}(q)|_{E^c_q}=(Dg^{m\pi})_q|_{E^c_q}=id$ holds, say in a $\xi$-neighborhood of $q$. Recall that, since $f_1$ has the shadowing property  $f_1^{m\pi}$ also has. Therefore, fixing $\epsilon\in(0,\frac{\xi}{4})$, there exists $\delta\in(0,\epsilon)$ such that every $\delta$-pseudo $f_1^{m\pi}$-orbit $(x_n)_n$ is $\epsilon$-traced by some point in $M$. Take $y$ such that $d(y,q)=\frac{3\xi}{4}$ and a closed $\delta$-pseudo $f_1^{m\pi}$-orbit $(x_n)_n$ such that any ball centered in $x_i$ and with radius $\epsilon$ is still contained in the $\xi$-neighborhood of $q$, moreover, take $x_0=q$ and $x_s=y$.

By the shadowing property there exists $z\in M$ such that $d(f_1^{m\pi i}(z),x_i)<\epsilon$ for any $i\in\mathbb{Z}$. Moreover, we observe that $d(f_1^{m\pi i}(z),q)<\xi$ for every $i\in\mathbb{Z}$. Therefore, $z\in E^c_q$. Finally, we reach a contradiction by noting that
$$d(q,z)\geq d(q,x_s)-d(x_s,z)=d(q,y)-d(x_s,f_1^{m\pi s}(z))\geq \frac{3\xi}{4}-\epsilon> \frac{\xi}{2}>\epsilon.$$

\end{proof}

\section{Volume-preserving diffeomorphisms}

Theorem~\ref{teo1} also holds on the broader context of volume-preserving diffeomorphisms. 

\begin{maintheorem}\label{teo2}
If a volume-preserving diffeomorphism $f$ defined in a manifold $M$ is $C^1$-stably shadowable, then $f$ is Anosov.  
\end{maintheorem}

\begin{proof}
The proof follows the same steps as the one of Theorem~\ref{teo1}. Next we give the strategy for proving it by referring the fundamental pieces that replace the ones used along the proof of Theorem~\ref{teo1}. 
\begin{itemize}
 \item The version of Proposition~\ref{Newhouse} for volume-preserving diffeomorphisms is proved in ~\cite[Theorem 1.1]{AC};
 \item Proposition~\ref{Moriyasu} can be obtained by noting that: 
\begin{enumerate}
 \item the Darboux coordinates are switched by the Moser volume-preserving coordinates (cf. ~\cite{Mo});
 \item the result in~\cite{Z} are replaced by the smoothness result in ~\cite{A};
\item the perturbation lemma in ~\cite{HT}, are interchanged by its correspondent in the volume-preserving case proved in \cite[Proposition 7.4]{BDP};
\item and, finally, we should use ~\cite[Theorem 3.6]{AM} instead of ~\cite[Theorem 3.9]{AM}.
\end{enumerate}

\end{itemize}
We leave the details to the reader.
 
\end{proof}

\section*{Acknowledgements}

The author was partially supported by the FCT-Funda\c{c}\~ao para a Ci\^encia e a Tecnologia, project PTDC/MAT/099493/2008.

\end{document}